\documentclass{iopart}

\usepackage{color, amssymb, amsthm,  amsfonts, amscd, epsfig}

\theoremstyle{plain}
\newtheorem{theo}{Theorem}[section]
\newtheorem{defi}[theo]{Definition}
\newtheorem{lem}[theo]{Lemma}
\newtheorem{prop}[theo]{Proposition}
\newtheorem{cor}[theo]{Corollary}
\newtheorem{rem}[theo]{Remark}

\newcommand{\R}{\mathbb R}

\newcommand{\C}{\mathbb C}
\newcommand{\Heis}{\mathbb H}
\newcommand{\B}{\mathcal B}
\newcommand{\Proj}{\mathcal P}
\newcommand{\Pp}{\mathbb{P}^1}

\newcommand{\F}{\mathcal F}
\newcommand{\Fock}{\mathbb F}
\newcommand{\PF}[2]{\Proj\Fock^{#1}_{#2}}

\newcommand{\Hil}{\mathcal H}
\newcommand{\fhil}{\mathcal{H}^\Omega}
\newcommand{\fprehil}{\stackrel{\circ}{\Hil}\!{}^\Omega}
\newcommand{\Sw}{\mathcal S}
\newcommand{\so}{\sim_\Omega}
\newcommand{\Exp}{\textrm{Exp}}
\newcommand{\fsn}[1]{\nmid #1 \nmid {}^\Omega}

\newcommand{\sn}[1]{\nmid #1 \nmid {}_\Omega}
\newcommand{\fspr}[2]{\langle #1 , #2 \rangle^{\Omega}}
\newcommand{\spr}[2]{\langle #1 , #2 \rangle_{\Omega}}

\newcommand{\fso}{\sim^\Omega}

\newcommand{\lra}{-\ \, \!\!\!\!\!\!\longrightarrow}

\begin{document}

\title[$SE(2)$ coherent states and the primary visual cortex]{Coherent states of the Euclidean group and activation regions of primary visual cortex}
\author{Davide Barbieri}
\address{Dipartimento di Matematica, Universit\`a di Bologna.\\ P.zza di P.ta S.Donato 5, 40126 Bologna, Italy.}
\ead{davide.barbieri8@unibo.it}
\author{Giovanna Citti}
\address{Dipartimento di Matematica, Universit\`a di Bologna.\\ P.zza di P.ta S.Donato 5, 40126 Bologna, Italy.}
\ead{giovanna.citti@unibo.it}
\author{Gonzalo Sanguinetti}
\address{Instituto de Ingenier\'ia El\'ectrica (IIE), Universidad de la Rep\'ublica.\\Montevideo, Uruguay.\\
DEIS, Universit\`a di Bologna.\\ Viale Risorgimento 2, 40136 Bologna, Italy.}
\ead{gsangui@fing.edu.uy}
\author{Alessandro Sarti}
\address{CREA \'Ecole Polytechnique/CNRS.\\ 32, boulevard Victor - 75015 Paris, France.}
\ead{alessandro.sarti@polytechnique.edu}

\begin{abstract}
The uncertainty principle of $SE(2)$ allows to construct a coherent states transform that is strictly related to the Bargmann transform for the group $\Heis^2$. The corresponding target space is characterized constructively and related to the almost complex structure of $SE(2)$ as a contact manifold. Such a coherent state transform provides a model for neural activity maps in the primary visual cortex, that are then described in terms of minimal uncertainty states. The results of the model are compared with the experimental measurements.
\end{abstract}

% \begin{abstract}
% Lorem ipsum dolor sit amet, consectetur adipisicing elit, sed do eiusmod tempor incididunt ut labore et dolore magna aliqua. Ut enim ad minim veniam, quis nostrud exercitation ullamco laboris nisi ut aliquip ex ea commodo consequat. Duis aute irure dolor in reprehenderit in voluptate velit esse cillum dolore eu fugiat nulla pariatur. Excepteur sint occaecat cupidatat non proident, sunt in culpa qui officia deserunt mollit anim id est laborum.
% \end{abstract}

\maketitle

% \tableofcontents

\section{Introduction}

Several notions of coherent states have been formulated that capture in great generality many aspects of this intrinsically interdisciplinary territory \cite{AAG}. Perhaps the most classical one when dealing with Lie group symmetries is that of Gilmore-Perelomov (GPCS) \cite{Perelomov,AAG}, that considers the orbits of some given fiducial vector under the action of irreducible unitary representations of the Lie group of interest. The main issue arising when this notion is applied to the Euclidean groups $E(n)$ is that the irreducible representations are not square integrable, hence no resolution of unity is provided by linear superposition on $L^2(E(n))$ of projectors, since this would give rise to divergence. This problem has been addressed in \cite{IK}, with the use of reducible representations constructed as direct integrals over finite intervals of representation parameters, while in \cite{AAG} the authors provide square integrable representations on the related homogeneous spaces as an application of a powerful result wich makes use of the geometry of coadjoint orbits for semidirect product groups. Moreover, in \cite{AAG} it is also pointed out that the construction of reproducing kernel Hilbert spaces (RKHS) is always possible for GPCS, independently on the square integrability, by restricting to the target of the corresponding coherent state transform.

We addressed the problem from a different perspective. Our approach is strongly motivated by the studies concerning the geometrization of the structure of the primary visual cortex (V1) in mammals \cite{Daugman,PT,CS}, the first cortical region involved in the processing of visual stimuli captured by the retinal receptors \cite{HW}. The functional architecture of V1 presents indeed the concurrency of two different symmetries, since the action of single neurons can be modeled as a linear filtering of retinal images with the (canonical) coherent states of the second Heisenberg group $\Heis^2$ \cite{Daugman,Ringach}, while the internal axonal connectivity can be described in terms of the Lie algebra of the group of Euclidean motions of the plane $SE(2)$ \cite{CS}. These two symmetries are tied together by the so-called orientation preference maps (OPM), that are mappings from the Euclidean plane $\R^2$ to the real projective line $\Pp$ defining at each point of V1, represented as a flat surface, the orientation of the Gabor filter corresponding to the relative cell \cite{Ohki}. These maps contain then informations on how coherent state analysis of two dimensional images is performed by V1 neurons using a two dimensional set of parameters instead of the total four dimensional set, due to the layered structure of V1, and have proven to be intimately related with the connectivity \cite{BZSF}.

The purpose of linking these two apparently unrelated symmetries has lead us to the GPCS notion of Euclidean coherent states, with fiducial vector chosen as a minimizer of the uncertainty principle for the irreducible representation of $SE(2)$. The corresponding RKHS can be related to the abstract construction explained in \cite{AAG}, but its concrete realization contains some extra structure which is relevant for the understanding of this specific problem. More precisely, the characterization of the space of surjectivity for the coherent state transform is similar to the one of the well known Bargmann-Fock space: an $L^2$ summability condition, but with respect to a singular measure, and a complex differentiability condition. This is not trivial, since the dimension of $SE(2)$ is odd, and hence it can carry only an almost complex structure \cite{Blair}: analyticity is then replaced by the weaker CR condition \cite{BER}. Then, while complex differentiability is enough to get surjectivity for $\Heis^2$, in this case the excess of redundancy in the coherent state transform needs also to be controlled by the measure. Moreover, this RKHS coincides with the target space of the canonical Bargmann transform when this last is restricted to irreducible Hilbert spaces for the representations of $SE(2)$. This result provides then a link between the two symmetries, and also motivates the CR condition, which naturally arises from the corresponding restriction of the ordinary Cauchy-Riemann equations.

In \cite{BCSS} we provided a model of the structure of OPM \cite{BG,Bl,BZSF} grounded on neurophysiological findings, which is able to reproduce neural activities of V1 measured in the in-vivo experiments and is in agreement with the proposed notion of coherent states. We would like to emphasize that this concrete application provides not only a motivation for the entire work, but also an example of a biological system whose remarkable organization can be deeply inspiring and demanding.

The paper is organized as follows. In Section \ref{Sec:coherent} we first provide precise definitions of the irreducible Hilbert spaces for the representations of $SE(2)$, following the classical approach introduced in \cite{V}. Then we define the natural coherent state transform acting on $L^2(S^1)$ and with Theorem \ref{teo:isosurj} we characterize its Hilbert space of surjectivity $\PF{\lambda}{\Omega}$. After a comparative discussion on the group structures of $\Heis^2$ and $SE(2)$, with Theorem \ref{Teo} and Corollary \ref{cor:diagram} we then explicit the functional relations between the two symmetries, showing that the newly defined transform acts as the projection of the classical Bargmann transform on $\PF{\lambda}{\Omega}$. In Section \ref{sec:model} we present, as the fundamental application, the model for activated regions of V1 introduced in \cite{BCSS}: after a description of the functional architecture of V1 and of the experimental setting, we precisely state the model in terms of the introduced coherent states and compare the results with the measurements.

% The paper is organized as follows. In Section \ref{Sec:coherent} we provide precise definitions of the irreducible Hilbert spaces for the representations of $SE(2)$, following the classical approach introduced in \cite{V}. Then we define the natural coherent state transform acting on $L^2(S^1)$ and in Theorem \ref{teo:isosurj} we characterize its Hilbert space of surjectivity $\PF{\lambda}{\Omega}$. After a comparative discussion on the group structures of $\Heis^2$ and $SE(2)$, we show in Theorem \ref{Teo} and Corollary \ref{cor:diagram} that the newly defined transform acts as the projection of the classical Bargmann transform on $\PF{\lambda}{\Omega}$. In Section \ref{sec:model} we present, as the fundamental application, the model for activated regions of V1 introduced in \cite{BCSS}: after a description of the functional architecture of V1 and of the experimental setting, we precisely state the model in terms of the introduced coherent states and compare the results with the measurements.

\section{Lie group symmetries and coherent states}\label{Sec:coherent}

\subsection{The group $SE(2)$ and the Heisenberg groups}

The Euclidean motion group $SE(2) = R^2_q \rtimes S^1_\theta$ is the noncommutative Lie group obtained as semidirect product between translation and counterclockwise rotations of the Euclidean plane with the usual composition law $(q',\theta') \cdot (q,\theta) = (q' + r_{\theta'}q, \theta' + \theta)$ \cite{V,Sugiura}. The left invariant vector fields can be calculated as \cite{CS}
\begin{equation}\label{generators}
X_1 = -\sin\theta\partial_{q_1} + \cos\theta\partial_{q_2} \ , \ X_2 = \partial_{\theta}
\end{equation}
whose commutator is given by
\begin{displaymath}
X_3 = [X_1,X_2] = \cos\theta\partial_{q_1} + \sin\theta\partial_{q_2}\ .
\end{displaymath}
Since the two vector fields (\ref{generators}) together with their commutator span the tangent space at any point, by Chow theorem \cite{M} every couple of points can be connected by curves that are piecewise Lie group exponential mappings
% Since the two vector fields (\ref{generators}) and their commutator span the tangent space at any point, any couple of points can be connected by piecewise exponential curves \cite{M}:
\begin{equation}\label{expSE2}
(q',\theta') = \Exp(t(X_1 + kX_2))(q,\theta)
\end{equation}
for some $t,k \in \R$, i.e. the group $SE(2)$ is naturally endowed with a so-called sub-Riemannian structure. The nonintegrable distribution of planes provided by $X_1$ and $X_2$ defines then a contact structure, associated to the contact form
\begin{equation}\label{contactform}
\omega = \cos\theta d q_1 + \sin\theta d q_2
\end{equation}
and $SE(2)$ can be seen as the double covering of the manifold of contact elements of the plane $\R^2 \times \Pp$. This last one is not orientable and hence can not carry a global contact form \cite{Blair}, but it is useful to note that it arises naturally as the projectivization of the four dimensional phase space \cite{Arnold}.

If we denote by $\Heis^n$ the n-th Heisenberg group \cite{Folland, M} in its semidirect product form $\R^n_q \rtimes (\R^n_p\times \R_t)$, defined by the group law
\begin{displaymath}
(p',q',t')\cdot(p,q,t) = (p'+p,q'+q,t'+t+p'q)
\end{displaymath}
we note that, by an argument analogous to the one expressed for $SE(2)$, we can associate to it a contact structure in the Darboux normal form
\cite{Arnold, Blair}
\begin{displaymath}
\omega_0 = p dq - dt
\end{displaymath}
in accordance with the notion of $\Heis^n$ as a central extension of the commutative $\R^{2n}$.

Darboux theorem tells then that locally the geometry of $SE(2)$ is that of $\Heis^1$, which indeed is its metric tangent cone \cite{Gromov}. Moreover, as we noted, the contact structure (\ref{contactform}) is directly inherited by the symplectic structure of the phase space $\R^{4}$, whose central extension returns the $\Heis^2$ group. This will be the point that will permit to relate, in subsection \ref{sec:SE2H2}, the groups $SE(2)$ and $\Heis^2$ in terms of the complex structures underlying their coherent states transforms.

\subsection{Coherent states of the Euclidean motion group}

We consider the coherent states of $SE(2)$ obtained with the Gilmore-Perelomov (GPCS) definition \cite{Perelomov,AAG}, starting from a minimum of the uncertainty principle. In order to do that, we make use of the irreducible representations of $SE(2)$ \cite{V,Sugiura}, and of the related algebra representation.
\begin{prop}
The action of the irreducible unitary representation of $SE(2)$ with parameter $\Omega \in \R$ on $u \in L^2(S^1)$ is given by
\begin{equation}\label{S1representation}
\Pi^\Omega(q,\theta) u(\varphi) = e^{-i\Omega (q_1\cos\varphi + q_2\sin\varphi)} u(\varphi - \theta)
\end{equation}
and for any $\phi_0 \in S^1$ we obtain a representation that is unitarily equivalent to (\ref{S1representation}) up to rotations, as
\begin{equation}\label{S1reprot}
\Pi^{\Omega_{\phi_0}}(q,\theta) u(\varphi) = e^{-i\Omega (q_1\cos(\varphi-\phi_0) + q_2\sin(\varphi-\phi_0))} u(\varphi - \theta)\ .
\end{equation}
\end{prop}

\begin{cor}
The representation of the Lie algebra of $SE(2)$ reads
\begin{equation}\label{eq:operators}
d\Pi^{\Omega_{\phi_0}} X_{1_0} = i\Omega\sin(\varphi - \phi_0) \quad , \quad d\Pi^{\Omega_{\phi_0}} X_{2_0} = \frac{d}{d\varphi}
\end{equation}
where $X_{1_0}$ and $X_{2_0}$ are the infinitesimal generators corresponding to the left invariant vector fields (\ref{generators}).
\end{cor}

Minimal uncertainty states for $SE(2)$ in the irreducible representation can be obtained as eigenvectors of the properly constructed annihilation operator \cite{Folland} in terms of the operators provided by (\ref{eq:operators}). This specific uncertainty principle has been discussed e.g. in \cite{CN,BreitUnc}, and we will be interested only in the eigenvector with eigenvalue zero.

\begin{prop}
The equation for minimal uncertainty states with zero average angular momentum \cite{CN} reads
\begin{equation}\label{eq:minunc}
\left(\frac{d}{d\varphi} + \lambda \Omega \sin(\varphi - \phi_0) \right) u^{\lambda,\Omega_{\phi_0}} (\varphi) = 0
\end{equation}
and is solved by
\begin{displaymath}
u^{\lambda,\Omega_{\phi_0}} (\varphi) = N e^{\lambda\Omega\cos(\varphi - \phi_0)}
\end{displaymath}
where $N$ is the $L^2(S^1)$ normalization. For $\phi_0 = 0$ we will simply denote them as $u^{\lambda,\Omega}$.
\end{prop}

GPCS can then be constructed starting from the fiducial vector $u^{\lambda,\Omega}$.
\begin{defi}
We consider the family of coherent states for the group $SE(2)$
\begin{equation}\label{eq:se2cs}
\psi^{\lambda,\Omega}(q,\theta;\varphi) \doteq \Pi^{\Omega}(q,\theta)u^{\lambda,\Omega}(\varphi)\ .
\end{equation}
\end{defi}

\subsection{The Bargmann transform for the two dimensional Euclidean group}

% This subsection is devoted to the analysis of the coherent state transform defined by the family (\ref{eq:se2cs}). In analogy with the classical Bargmann transform, and because of the result expressed by Theorem \ref{Teo}, we will call it $SE(2)$-Bargmann transform. Its geometric properties are contained in the structure of its Hilbert space of surjectivity, wich we will accordingly call $SE(2)$-Bargmann space, and denote it $\PF{\lambda}{\Omega}$. This is the space of functions on the $SE(2)$ group that supply an $L^2$ summability condition and a complex differentiability condition, as for the ordinary Bargmann space \cite{Folland}. With respect to the variables representing translations, the summability is expressed in terms of a Hilbert space, that we will denote $\Hil_\Omega$, which roughly speaking consists of functions whose Fourier transform is concentrated on a circle, and hence can be treated as functions in $L^2(S^1)$, following the classical strategy used to characterize the irreducible Hilbert space for the representations of $SE(2)$ \cite{V}. The complex differentiability relies instead on the quasicomplex structure that can be associated to $SE(2)$ as a contact manifold \cite{Blair}. This condition will tell that $\PF{\lambda}{\Omega}$ is a space of CR functions \cite{BER}, which amounts to say that these functions are eigenfunction of the annihilation operator used to obtain minimal uncertainty states, but expressed in terms of the differential operators (\ref{generators}), and not in terms of their representation (\ref{eq:operators}).

This subsection is devoted to the analysis of the coherent state transform defined by the family (\ref{eq:se2cs}), which we will call $SE(2)$-Bargmann transform in analogy with the classical Bargmann transform (see also Theorem \ref{Teo}). Its Hilbert space of surjectivity, denoted $\PF{\lambda}{\Omega}$ is the space of functions on $SE(2)$ that satisfy an $L^2$ summability condition and a complex differentiability condition, as for the ordinary Bargmann space \cite{Folland}. Following the classical strategy used in \cite{V} to characterize the irreducible Hilbert spaces for the representations of $SE(2)$, the summability with respect to the $\R^2$ variables is expressed in terms of a Hilbert space $\Hil_\Omega$ which, roughly speaking, consists of functions whose Fourier transform is concentrated on a circle, and hence can be treated as functions in $L^2(S^1)$. The complex differentiability relies instead on the almost complex structure that can be associated to $SE(2)$ as a contact manifold \cite{Blair}, and tells that $\PF{\lambda}{\Omega}$ is a space of CR functions \cite{BER}.

\begin{defi}
We will call $SE(2)$-Bargmann transform of a function $\Phi \in L^2(S^1)$
\begin{equation}\label{eq:SE2Barg}
\B_{\Omega}^{\lambda}\Phi(q,\theta) \doteq \langle \Pi^{\Omega}(q,\theta) u^{\lambda,\Omega}, \Phi\rangle_{L^2(S^1)}
\end{equation}
\end{defi}

We start now the construction of the Hilbert space $\Hil_\Omega$. In what follows we will use for the (unitary) Fourier transform the convention
\begin{displaymath}
\F f (k) = \frac{1}{2\pi}\int_{\R^2} e^{-ik \cdot x} f(x) dx\ .
\end{displaymath}

\begin{defi}\label{def:fsemin}
Let $f \in \Sw(\R^2)$ be a function in the Schwartz class and call $\hat f \in \Sw(\hat\R^2)$ its Fourier transform. We define the distributions
\begin{eqnarray*}
\hat f^\Omega(k) \doteq \hat f(k) \,\frac{1}{\Omega}\delta(|k| - \Omega) \in \Sw'(\hat\R^2)\\
f_\Omega (x) \doteq \frac{1}{(2\pi)^2}f \ast j_0(\Omega |\cdot|)(x) = \int_{\R^2} f(y) \, j_0(\Omega |x - y|) \frac{dy}{(2\pi)^2}  \in \Sw'(\R^2)
\end{eqnarray*}
where $j_0$ stands for the Bessel function of order zero \cite{V} $j_0(s) = \int_0^{2\pi} e^{i s \cos(\varphi)} d\varphi$.\\
We will also call the corresponding operators from $\Sw$ to $\Sw'$
\begin{displaymath}
\Proj^\Omega \hat{f} (k) = \hat f^\Omega(k) \quad ; \quad \Proj_\Omega f (x) = f_\Omega(x)\ .
\end{displaymath}
\end{defi}

\begin{defi}
By $\fsn{\cdot}$ we denote the seminorm on $\Sw(\hat\R^2)$
\begin{equation}\label{eq:fsemin}
\fsn{\hat f} \doteq (\langle \hat f^\Omega , \hat f\rangle_{\Sw'\,\Sw})^{1/2}
\end{equation}
noting that $\fsn{\hat f} = \left(\int_{\R^2} |\hat f(k)|^2 d\mu_\Omega(k)\right)^{1/2}$, where $d\mu_\Omega(k) = \frac{1}{\Omega} \delta(|k| - \Omega) dk$, 
and by $\sn{\cdot}$ we denote the functional on $\Sw(\R^2)$
\begin{equation}\label{semin}
\sn{f} \doteq \left(\langle f_\Omega , f\rangle_{\Sw'\,\Sw}\right)^{1/2} .
\end{equation}
\end{defi}

The introduced operators $\Proj$ and functionals $\nmid \cdot \nmid$ are related by distributional Fourier transform, as expressed by the following lemma.

\begin{lem}\label{lem:keyreal}
The following hold
\begin{itemize}
\item[i)] $\Proj_\Omega = \F^{-1} \Proj^\Omega \F$ in distributional sense
\item[ii)] $\sn{f} = \,\fsn{\hat f}$ for all $f \in \Sw(\R^2)$
\end{itemize}
\end{lem}
\begin{proof}
The first claim reads equivalently $\F f_\Omega = \hat f^\Omega$. To this end we only need to show
\begin{equation}\label{dum1}
\left(\F j_0(\Omega |\cdot|)\right)(k) = \frac{2\pi}{\Omega}\delta(|k|-\Omega)
\end{equation}
since by standard arguments on tempered distributions
\begin{displaymath}
\left(\F f_\Omega\right)(k) = \frac{1}{(2\pi)^2}\left(\F f \ast j_0(\Omega |\cdot|)\right)(k) = \frac{1}{2\pi} \hat f(k) \left(\F j_0(\Omega |\cdot|)\right)(k)\ .
\end{displaymath}
To see (\ref{dum1}), take $\psi\in\Sw(\R^2)$
$$
\begin{array}{rcl}
\displaystyle{\frac{2\pi}{\Omega}} \langle \delta(|k| \!\!\!& - &\!\!\! \Omega),\F \psi \rangle_{\Sw'\Sw} \ = \ 2\pi\displaystyle{\int_{0}^{2\pi} \hat\psi(\Omega\cos\varphi,\Omega\sin\varphi) d\varphi}\vspace{4pt}\\
& = & \displaystyle{\int_{0}^{2\pi} d\varphi\int_{\R^2}dx\psi(x) e^{-i\Omega|x|\cos(\varphi - \alpha_x)}
 = \int_{\R^2}dx\psi(x) \overline{\int_{0}^{2\pi} d\varphi e^{i\Omega|x|\cos\varphi}}}\ .
\end{array}
$$
The second claim is then a consequence of Definition \ref{def:fsemin} and the distributional Parseval theorem, since
$\langle f_\Omega , f\rangle_{\Sw'\,\Sw} = \langle \F f_\Omega , \F f\rangle_{\Sw'\,\Sw}$.
\end{proof}

\begin{defi}
Let $\fso$ be the equivalence relation induced by (\ref{eq:fsemin})
\begin{displaymath}
\hat f_1 \fso \hat f_2 \quad \Leftrightarrow \quad \fsn{\hat f_1 - \hat f_2} = 0
\end{displaymath}
and $[\cdot]^\Omega$ the corresponding equivalence classes. We call $\fprehil$ the space $\fprehil = \Sw(\hat\R^2)/\fso$
and denote by $[\hat f]^\Omega$ its elements.
\end{defi}

By definition, then, we have that $\fsn{\cdot}$ is a norm on $\fprehil$. Moreover, since this quotient keeps only informations on the behavior of functions on the circle of radius $\Omega$, then the elements of $\fprehil$ can be considered as functions of the polar angle.

\begin{lem}\label{lem:fhilS1}
$\fhil \doteq \overline{\fprehil \, }^{\,{}_{\fsn{\cdot}}} \approx L^2(S^1)$
\end{lem}
\begin{proof}
Any element $[\hat f]^\Omega \in \fprehil$ is represented by a function of the polar angle of $k$ as
$[\hat f]^\Omega(\varphi) = \hat f(\Omega\cos\varphi,\Omega\sin\varphi) \in \Sw(S^1)$\ .
% Calling $\varphi$ the polar angle of $k$, we observe that any element $[\hat f]^\Omega \in \fprehil$ is represented by a function on $S^1$ as
% $[\hat f]^\Omega(\varphi) = \hat f(\Omega\cos\varphi,\Omega\sin\varphi) \in \Sw(S^1)$\ .
% \begin{equation}\label{eq:fclass}
% [\hat f]^\Omega(\varphi) = \hat f(\Omega\cos\varphi,\Omega\sin\varphi) \ \in \ \Sw(S^1)\ .
% \end{equation}
% since clearly
% \begin{displaymath}
% \hat f_1, \hat f_2 \in [\hat f]^\Omega \iff \hat f_1(\Omega\cos\varphi,\Omega\sin\varphi) = \hat f_2(\Omega\cos\varphi,\Omega\sin\varphi) \ \forall \varphi \in S^1 .
% \end{displaymath}
In particular, 
\begin{equation}\label{eq:frep}
\hat f^\Omega(k) = [\hat f]^\Omega(\varphi) \frac{1}{\Omega}\delta(|k| - \Omega) .
\end{equation}
Moreover, we have that $\fsn{\hat f} = \|[\hat f]^\Omega\|_{L^2(S^1)}$.
% , indeed
% \begin{eqnarray*}
% \fsn{\hat f} & = & \int_0^\infty d\omega\, \omega \int_0^{2\pi} d\varphi |\hat f(\omega\cos\varphi,\omega\sin\varphi)|^2 \frac{1}{\Omega} \delta(\omega - \Omega)\\
% & = & \int_0^{2\pi} d\varphi |\hat f(\Omega\cos\varphi,\Omega\sin\varphi)|^2 = \int_0^{2\pi} d\varphi |[\hat f]^\Omega(\varphi)|^2.
% \end{eqnarray*}
The density of the quotient space relies then on the ordinary density result of $\Sw(S^1)$ in $L^2(S^1)$.
\end{proof}

By Lemma \ref{lem:keyreal} and Lemma \ref{lem:fhilS1}, also (\ref{semin}) is a seminorm.

\begin{cor}
The functional (\ref{semin}) is a seminorm on $\Sw(\R^2)$ and the corresponding equivalence classes are formed by functions that are equal when convolved with the Bessel function $j_0$. More precisely
\begin{displaymath}
\sn{f_1 - f_2} = 0 \ \Leftrightarrow \ (f_1 - f_2)_\Omega = 0
\end{displaymath}
and in this case we write $f_1 \so f_2$ and indicate by $[.]_\Omega$ the related equivalence classes.
\end{cor}

\begin{defi}
We will call $\Hil_\Omega$ the closure of the quotient space $\Sw(\R^2)/\so$ with respect to the norm $\sn{.}$ and denote by $[f]_\Omega$ its elements.
\end{defi}

Lemma \ref{lem:keyreal} can now be applied to the introduced Hilbert spaces.

\begin{cor}\label{cor:Hilfourier}
The spaces $\Hil_\Omega$ and $\fhil$ are Fourier transformed of one another. More precisely the Fourier transform
\begin{displaymath}
\F : \ \Hil_\Omega \, \longrightarrow \,\fhil
\end{displaymath}
is a bijection, isometric with respect to the norms $\sn{.}$ and $\fsn{.}$.
\end{cor}

\begin{rem}
The spaces $\Hil_\Omega$ and $\fhil$ are Hilbert spaces. Indeed, for $\fhil$ the natural scalar product $\fspr{.}{.}$ is the one inherited from  $L^2(S^1)$ while for $\Hil_\Omega$ we can define
\begin{displaymath}
\spr{[f]_\Omega}{[g]_\Omega} \doteq \langle [\hat{f}]^\Omega , [\hat{g}]^\Omega\rangle_{L^2(S^1)} = \langle f_\Omega , g\rangle_{\Sw'\Sw}
\end{displaymath}
where the last transition is due to Corollary \ref{cor:Hilfourier}.
\end{rem}

We recall that by the standard arguments in \cite{V,Sugiura} and the theory of direct integral for $SE(2)$, Lemma \ref{lem:fhilS1} and Corollary \ref{cor:Hilfourier} provide the following.

\begin{prop}\label{prop:inv}
The Hilbert spaces $\Hil_\Omega$ and $\fhil$ are $SE(2)$ invariant.
\end{prop}

Working with equivalence classes allows some useful flexibility, but should be handled carefully. In order to avoid confusion, we explicitly state this corollary.

\begin{cor}\label{Cor:trick}
Let $f$ be in $L^2(\R^2)$ and $\Phi$ in $L^2(S^1)$. Then
\begin{itemize}
\item[i)] an element $[f]_\Omega$ of $\Hil_\Omega$ can be represented as an $\Sw'$ distribution by $f_\Omega$
\item[ii)] we can unambiguously extend the notation $[\cdot]^\Omega$ to $\Sw'$ distributions of type $T = \Phi(\varphi) \frac{1}{\Omega}\delta(|k|-\Omega)$, meaning $[T]^\Omega = \Phi(\varphi)$.
\end{itemize}
\end{cor}

\begin{proof}[Proof of ii)]
To any distribution $T = \Phi(\varphi) \frac{1}{\Omega}\delta(|k|-\Omega)$ we can indeed associate a function $\hat{f}_T \in L^2(\hat\R^2)$ such that $[\hat{f}_T]^\Omega(\varphi) = \Phi(\varphi)$, simply providing one $L^2$ prolongation outside the circle, e.g. $\hat{f}_T(k) = \Phi(\varphi)g(|k|)$. So $ii)$ amounts to say that we indicate for simplicity $[T]^\Omega$ to mean $[\hat{f}_T]^\Omega$. 
\end{proof}

% \begin{lem}\label{lem:realdecomp}
% Let $f \in \Sw(\R^2)$. Then the following representation formula
% holds
% \begin{equation}\label{eq:conv}
% f(x) = \int_0^\infty d\Omega\, \Omega\, f_\Omega(x) .
% \end{equation}
% \end{lem}
% \begin{proof}
% \begin{eqnarray*}
% f(x) & = & \int_{\R^2} d \omega \hat f(\omega) e^{ix\omega}   =
% \int_{\R^2} d \omega \int_{\R^2} dz f(z)   e^{i(x-y)\omega} \\
% & = & \int_0^\infty d\Omega\, \Omega\, \int_0^{2\pi}d\varphi\, \int_{\R^2} dy\, f(x-y)  e^{i\Omega |y| \cos(\varphi - \alpha_y)}\\
% & = & \int_0^\infty d\Omega\, \Omega\, \int_{\R^2} dy\, f(x-y) \int_0^{2\pi}d\varphi\, e^{i\Omega |y| \cos(\varphi)}\\
% & = & \int_0^\infty d\Omega\, \Omega \int_{\R^2} dy f(x-y)
% J_0(\Omega|y|)= \int_0^\infty d\Omega\, \Omega f_\Omega(x)
% \end{eqnarray*}
% \end{proof}

These notions allow us to define the Hilbert space of surjectivity for the $SE(2)$-Bargmann transform.
\begin{defi}
Let us call
$$
\begin{array}{rcl}
\mathcal{H}_\Omega(\R^2,S^1) & \doteq & \left\{F:\R^2_q\times S^1_\theta \rightarrow \C \ \textrm{such that} \ \theta \mapsto F(q,\theta) \ \textrm{is in} \ L^2(S^1)\right.\\
&& \ \, \left.\textrm{and} \ q \mapsto F(q,\theta) \ \textrm{is in} \ \mathcal{H}_\Omega\right\}
\end{array}
$$
and let us say that a function $F : \R^2_q \times S^1_\theta \rightarrow \C$ is $CR^\lambda$ if
\begin{equation}\label{CRlambda}
\left(X_2 - i \lambda X_1\right) F = 0
\end{equation}
where $X_1$ and $X_2$ are the left invariant differential operators given by (\ref{generators}).\\
We define the $SE(2)$-Bargmann space as
\begin{equation}\label{eq:eFock}
\PF{\lambda}{\Omega} = \Hil_\Omega(\R^2_q \times S^1_\theta) \cap CR^\lambda\ .
\end{equation}
\end{defi}

On the basis of this definition we can indeed prove the following.
\begin{theo}\label{teo:isosurj}
$\B_{\Omega}^{\lambda}: L^2(S^1) \rightarrow \PF{\lambda}{\Omega}$ is an isometric surjection.
\end{theo}
\begin{proof}
By direct computation we can see that the functions $\B_{\Omega}^{\lambda}\Phi(q,\theta)$ satisfy equation (\ref{CRlambda}) for any $\Phi \in L^2(S^1)$.
\newpage\noindent
Moreover, the distributional Fourier transform of
\begin{eqnarray}\label{dum2}
q \mapsto \B_{\Omega}^{\lambda}\Phi(q,\theta) & = & \langle \Pi^{\Omega}(\theta,q)u^{\lambda,\Omega} , \Phi \rangle_{L^2(S^1)}\\
& = & N \int_{S^1} d\varphi e^{i\Omega(q_1\cos\varphi + q_2\sin\varphi)}e^{\lambda \Omega \cos(\varphi - \theta)} \Phi(\varphi)\nonumber
\end{eqnarray}
can be calculated as for Lemma \ref{lem:keyreal} i), providing
\begin{equation}\label{eq:FBSE2}
\F \B_{\Omega}^{\lambda}\Phi (k,\theta) = N e^{\lambda\Omega\cos(\theta - \varphi)} \Phi(\varphi) \frac{1}{\Omega}\delta(|k| - \Omega)
\end{equation}
hence (\ref{dum2}) is the $\Sw'$ representation of a $\mathcal{H}_\Omega$ function, as from Corollary \ref{Cor:trick} i).

The isometry can now be proved since
\begin{displaymath}
\left[\F \B_{\Omega}^{\lambda}\Phi\right]^\Omega (\varphi,\theta) = N e^{\lambda\Omega\cos(\theta - \varphi)} \Phi(\varphi)
\end{displaymath}
as in Corollary \ref{Cor:trick} ii), and
\begin{displaymath}
\int_{S^1} d\varphi \int_{S^1} d\theta \left|\left[\F \B_{\Omega}^{\lambda}\Phi\right]^\Omega (\varphi,\theta)\right|^2 = \int_{S^1} |\Phi(\varphi)|^2 d\varphi\ .
\end{displaymath}

To prove surjectivity, consider a function $F(q,\theta)$ satisfying (\ref{CRlambda}). Then its Fourier transform with respect to the $q$ variables must satisfy
\begin{equation}\label{CRf}
\left(\partial_\theta + \lambda \kappa \sin(\theta - \varphi)\right) \F F(\kappa\cos\varphi,\kappa\sin\varphi,\theta) = 0\ .
\end{equation}
hence
\begin{displaymath}
\F F(\kappa\cos\varphi,\kappa\sin\varphi,\theta) = e^{\lambda \kappa \cos(\theta - \varphi)} \Phi(\varphi) g(\kappa)
\end{displaymath}
for some $\Phi(\varphi), g(\kappa)$. But $F \in \mathcal{H}_\Omega(\R^2,S^1)$, so $\Phi \in L^2(S^1)$ and $g(\kappa) = c \delta(\kappa - \Omega)$.
\end{proof}

We remark that equation (\ref{CRf}) is the same as equation (\ref{eq:minunc}), so the CR property on this space is a minimal uncertainty condition.

As a corollary, we can express the inversion of the $SE(2)$-Bargmann transform in a familiar way.

\begin{cor}\label{cor:SE2inv}
The following holds
\begin{displaymath}
\Phi(\varphi) = \left[\int_{L^2(\R^2\times S^1)} \Pi^{|k|}(\theta,q)u^{\lambda,|k|}(\varphi) \langle \Pi^{\Omega}(\theta,q)u^{\lambda,\Omega} , \Phi \rangle_{L^2(S^1)} \frac{dq d\theta}{2\pi}\right]^\Omega
\end{displaymath}
\end{cor}
\begin{proof}
By (\ref{eq:FBSE2}) we have
$$
\begin{array}{rcl}
\Phi(\varphi) \frac{1}{\Omega}\delta(|k| - \Omega)
& = &N\displaystyle{\int_0^{2\pi} d\theta e^{\lambda \Omega \cos(\varphi - \theta)} \F \B_{\Omega}^{\lambda}\Phi (k,\theta)}\vspace{4pt}\\
& = &\frac{N}{2\pi}\displaystyle{\int_0^{2\pi} d\theta \int_{\R^2} dq e^{-i|k|(q_1\cos\varphi + q_2\sin\varphi)}e^{\lambda \Omega\cos(\varphi - \theta)} \B_{\Omega}^{\lambda}\Phi (q,\theta)}\vspace{4pt}\\
& = &\displaystyle{\int_{L^2(\R^2\times S^1)} \Pi^{|k|}(\theta,q)u^{\lambda,|k|}(\varphi) \langle \Pi^{\Omega}(\theta,q)u^{\lambda,\Omega} , \Phi \rangle_{L^2(S^1)} \frac{dq d\theta}{2\pi}}\ . 
\end{array}
$$
By Corollary \ref{Cor:trick} ii) we can then conclude.
\end{proof}

This result shows explicitly that the choice of the equivalence class does not amount to simply set $|k| = \Omega$ when applied to $\Sw'$ distributions, but rather allows to overcome the divergence that this position would give. This divergence is indeed the conterpart of the $L^2(SE(2))$ nonintegrability discussed in \cite{IK}. In the end, the precise definition of the space $\Hil_\Omega$ provides a constructive way to realize the summability condition for the RKHS related to the $SE(2)$-Bargmann transform that would result following the approach described in \cite{AAG}.

\subsection{Connections with the Heisenberg group $\Heis^2$}\label{sec:SE2H2}

% This subsection aims to provide a relation between the symmetries defined by $SE(2)$ and $\Heis^2$. This can be done at the level of coherent states transforms and the related complex structures. As we have seen in the previous section, the $SE(2)$-Bargmann space is characterized by the almost complex structure inherited from the contact structure of $SE(2)$. In the case of $\Heis^2$, the Bargmann space is related to the complex structure of $\R^4$ that provides the analiticity of the Bargmann transform. This is precisely the one related to the symplectic structure that induces the contact structure of $SE(2)$. As it will be clear with Proposition \ref{prop:minuncmix}, the almost complex structure related to the $SE(2)$-Bargmann space is indeed the one induced on $\R^2 \times S^1$ considered as a real submanifold of $\C^2$ \cite{BER}. The center of $\Heis^2$, which prevents any direct relation between the two groups, does not interfere at the level of coherent states, since it behaves as a phase factor and is usually disregarded in favor of the usage of the so-called projective (Schr\"odinger) representation \cite{Folland,Perelomov,AAG}. Nevertheless, it is $\Heis^2$ noncommutativity that allows to obtain the Bargmann transform and, in the end, the analiticity condition, which is the counterpart of the information redundancy in the coherent states transform imposed by uncertainty principle.

This subsection provides a relation between the symmetries defined by $SE(2)$ and $\Heis^2$ in terms of their coherent states transforms and the related complex structures. As we have seen, the $SE(2)$-Bargmann transforms are CR functions in the almost complex structure of $SE(2)$. On the other hand, it is well known that the Bargmann transform is analytic with respect to the complex structure of $\R^4$, that corresponds precisely to the symplectic structure that induces the contact structure of $SE(2)$. As it will be clear with Proposition \ref{prop:minuncmix}, the almost complex structure related to the $SE(2)$-Bargmann space is indeed the one induced on $\R^2 \times S^1$ considered as a real submanifold of $\C^2$ \cite{BER}. The center of $\Heis^2$, which prevents any direct relation between the two groups, does not interfere at the level of coherent states, since it behaves as a phase factor and is usually disregarded in favor of the usage of the so-called projective (Schr\"odinger) representation \cite{Folland,Perelomov,AAG}.

% Nevertheless, it is $\Heis^2$ noncommutativity that allows to obtain the Bargmann transform and, in the end, the analiticity condition, which is the counterpart of the information redundancy in the coherent states transform imposed by uncertainty principle.

For notational convenience we recall the following two classical results.

\begin{defi}
For any $\sigma \in \R^+$, we obtain a family of minimal uncertainty coherent states for the Heisenberg group $\Heis^2$ on $L^2(\R^2)$ as
\begin{equation}\label{Hcs}
\psi^\sigma_{q,p}(x) = M e^{ip\cdot (x - q) } e^{-\frac{|x - q|^2}{2\sigma^2}} \ , \quad (q,p) \in \R^4
\end{equation}
where $M$ is the $L^2(\R^2)$ normalization. In what follows we will indicate the space-frequency analysis provided by the $\Heis^2$ coherent states transform as
\begin{equation}\label{eq:csa}
F(q,p) = \langle \psi^\sigma_{q,p} , f\rangle_{L^2(\R^2)}\ .
\end{equation}
\end{defi}

\begin{prop}\label{BH2}
The Bargmann transform of $f \in L^2(\R^2)$, defined as
$\B^{\Heis^2}f(q,p) = e^{\frac{\sigma^2 |p|^2}{2}} F(q,p)$
is an isometry onto the Bargmann space
\begin{displaymath}%\label{Fock}
\Fock_\sigma = L^2(\R^2_q \times \R^2_p, e^{-\frac{\sigma^2 |p|^2}{2}}dqdp) \cap Hol(\C^2_{z_1,z_2})
\end{displaymath}
where holomorphy is intended with respect to the complex structure $z_j = (q_j,\sigma^2 p_j)$
\begin{equation}\label{analytic}
\left(\partial{p_j} + i\sigma^2\partial{q_j}\right) \B^{\Heis^2}f = 0 \quad , \quad j=1,2 \ .
\end{equation}
\end{prop}

The simple observation that leads the argument consists in inspecting the frequency behavior of functions (\ref{Hcs}) in polar coordinates. Setting
\begin{equation}\label{notationlambda}
\sigma^2 p = \lambda(\cos\theta,\sin\theta)
\end{equation}
the localization in frequency space reads
\begin{displaymath}
e^{-\frac{|k - p|^2\sigma^2}{2}} = e^{-\frac{|k|^2\sigma^2}{2}}e^{-\frac{\lambda^2}{2\sigma^2}}e^{\lambda\,|k|\,\cos(\varphi - \theta)}
\end{displaymath}
so in particular angular localization satisfies minimal uncertainty (\ref{eq:minunc}).

\begin{prop}\label{prop:minuncmix}
Let $F(q,p)$ be given by (\ref{eq:csa}), and let $U(k,p)$ be its spatial Fourier transform $U(k,p) = \F_{qk} F (k,p)$. By setting both variables in polar coordinates, call $U^{\lambda,\kappa}(\varphi,\theta) = U(\kappa\cos\varphi,\kappa\sin\varphi,\frac{\lambda}{\sigma^2}\cos\theta,\frac{\lambda}{\sigma^2}\sin\theta)$. Then the functions $\theta \mapsto U^{\lambda,\kappa}(\varphi,\theta)$ are minimal uncertainty states for $SE(2)$ in the $\varphi$-rotated representation ({\ref{S1reprot}}).
\end{prop}
\begin{proof}
Any $\C^2$-analytic function restricted to the real submanifold obtained by fixing the modulus of $p$ is a CR function \cite{BER} with respect to the vector fields that generate the group $SE(2)$. Specifically, from the Cauchy-Riemann equations (\ref{analytic}) it follows that
\begin{equation}\label{CR}
\left(X_2 - i \sigma^2 |p| X_1\right) \B^{\Heis^2}f = 0 \ .
\end{equation}
We note that since no derivation is performed along $|p|$, this equation equivalently holds also for (\ref{eq:csa}), and we are in the same situation as in the proof of Theorem \ref{teo:isosurj}, holding equation (\ref{CRf}).
\end{proof}

Proposition \ref{prop:minuncmix} motivates the following.

\begin{theo}\label{Teo}
Let us denote by ${}^-\B^{\Heis^2}$ the linear extension of $\B^{\Heis^2}$ to tempered distributions $\mathcal{S}'$ (see \cite{Bargmann}), and let $f \in L^2(\R^2)$ then
\begin{displaymath}%\label{eq:BB}
{}^-\B^{\Heis^2} f_\Omega (q,p) = c \,\B_{\Omega}^{\lambda}[\hat{f}]^\Omega(q,\theta)
\end{displaymath}
where the constant $c$ is given by $c = \frac{\sqrt{j_0(-2 i \lambda \Omega)}}{\sigma\sqrt{\pi}}e^{-\frac{\sigma^2 \Omega^2}{2}}$.
\end{theo}
\begin{proof}
We prove the result first by showing that the Bargmann transform acts as a composition of $SE(2)$-Bargmann transforms, and then projecting to an irreducible subspace. Calling $\hat{f}^{(\kappa)}(\varphi) = \hat{f}(\kappa\cos\varphi,\kappa\sin\varphi)$, and according to (\ref{notationlambda})
$$
\begin{array}{rcl}
\B^{\Heis^2}f(q,p) & = & e^{\frac{\sigma^2 |p|^2}{2}} \langle \psi^\sigma_{q,p} , f\rangle_{L^2(\R^2)} \ = \ e^{\frac{\sigma^2 |p|^2}{2}} \langle \F \psi^\sigma_{q,p} , \F f\rangle_{L^2(\R^2)}\vspace{4pt}\\
& = & M \displaystyle{e^{\frac{\sigma^2 |p|^2}{2}} \int_{\R^2} dk e^{ik\cdot q} e^{-\frac{\sigma^2|k-p|^2}{2}} \hat{f}(k)}\vspace{4pt}\\
& = & M \displaystyle{\int_0^{+\infty} d\kappa \kappa e^{-\frac{\sigma^2 \kappa^2}{2}}\int_{S^1} d\varphi e^{i\kappa(q_1\cos\varphi + q_2\sin\varphi)}e^{\lambda \kappa\cos(\varphi - \theta)} \hat{f}^{(\kappa)}(\varphi)}\vspace{4pt}\\
& = & \displaystyle{\frac{M}{N}\int_0^{+\infty} d\kappa \kappa e^{-\frac{\sigma^2 \kappa^2}{2}} \langle \Pi^{\kappa}(\theta,q)u^{\lambda,\kappa} , \hat{f}^{(\kappa)} \rangle_{L^2(S^1)}}
\end{array}
$$
and the normalization constants are $M = \frac{1}{\sigma\sqrt{\pi}}$ and $N = \frac{1}{\sqrt{j_0(-2 i \lambda \Omega)}}$.

The proof is concluded since if we perform the Bargmann transform on $f_\Omega$, then by Lemma \ref{lem:keyreal} i) $\F f_\Omega(k) = \hat{f}^{\Omega}(k) = [\hat{f}]^\Omega(\varphi)\frac{1}{\Omega}\delta(\kappa - \Omega)$.
\end{proof}

The relations among the various Hilbert spaces is now summarized.

\begin{cor}\label{cor:diagram}
The following diagram is commutative
% $$
% \begin{array}{ccc}
% L^2(\R^2_x) & \stackrel{\B^{\Heis^2}}{\lra} & \Fock_\sigma \vspace{6pt}\\
% \Big\downarrow \ \scriptstyle{\Proj_\Omega} & & \quad \Big\downarrow \ \scriptstyle{\Proj_\Omega}\vspace{6pt}\\
% \Hil_\Omega(\R^2_x) & & \\
% & \searrow\,\scriptstyle{\B^{\Heis^2}} & \\
% \Big\downarrow \ \scriptstyle{\F} & & \ \PF{\lambda}{\Omega} \\
% & \nearrow\scriptstyle{\B_\Omega^\lambda} & \\
% \fhil(\hat\R^2_k) \approx L^2(S^1_\varphi) & & 
% \end{array}
% $$
$$
\begin{array}{ccccl}
L^2(\R^2_x) \ & \stackrel{\Proj_\Omega}{\lra} & \Hil_\Omega(\R^2_x) & \stackrel{\F}{\lra} & \fhil(\hat\R^2_k) \approx L^2(S^1_\varphi)\vspace{6pt}\\
\Big\downarrow \ \scriptstyle{\B^{\Heis^2}} & & \qquad \searrow \ \scriptstyle{\B^{\Heis^2}} & & \swarrow\scriptstyle{\B_\Omega^\lambda} \vspace{6pt}\\
\Fock_\sigma \quad & \stackrel{\Proj_\Omega}{\lra} & & \PF{\lambda}{\Omega} \
\end{array}
$$
where the map $\Hil_\Omega(\R^2_x) \stackrel{\B^{\Heis^2}}{\longrightarrow} \PF{\lambda}{\Omega}$ is intended in the sense of Theorem \ref{Teo} and $\B^{\Heis^2}$ transforms with respect to the $p$ variable are considered as functions of the polar angle.
\end{cor}
\begin{proof}
The only relation that has not been inspected is $\Proj_\Omega: \Fock_\sigma \rightarrow \PF{\lambda}{\Omega}$. The $CR^\lambda$ regularity is due to (\ref{CR}), while $\left[\B^{\Heis^2}f(q,p)\right]_\Omega$, where the equivalence class is considered with respect to $q$, belongs to $\Hil_\Omega(\R^2_q\times S^1_\theta)$, where $\theta$ is the polar angle of $p$. The commutativity (up to a constant factor), can be seen by direct computation. Denoting with $\F_{qk}$ the Fourier transform with respect to spatial variables
\begin{eqnarray*}
\Proj^\Omega \F_{qk} \left(\B^{\Heis^2}f\right)(k,p) & = & 2\pi\sigma^2 M \ \Proj^\Omega \hat{f}(k) e^{-\frac{|k|^2\sigma^2}{2}} e^{\sigma^2|p|\,|k|\cos(\theta - \varphi)}\\
% & = & 2\sqrt{\pi}\sigma \ e^{-\frac{\Omega^2\sigma^2}{2}} \ [\hat{f}]^\Omega e^{\sigma^2|p|\,\Omega\cos(\theta - \varphi)}\frac{1}{\Omega}\delta(|k|-\Omega)\\
& = & 2cN \ [\hat{f}]^\Omega e^{\sigma^2|p|\,\Omega\cos(\theta - \varphi)}\frac{1}{\Omega}\delta(|k|-\Omega)\ .
\end{eqnarray*}
Then, by (\ref{eq:FBSE2}) and Theorem \ref{Teo}
\begin{displaymath}
\Proj^\Omega \F_{qk} \left(\B^{\Heis^2}f\right)(k,p) = 2\F_{qk} \left(\B^{\Heis^2}f_\Omega\right)(k,p)
\end{displaymath}
hence using Lemma \ref{lem:keyreal} we have $\Proj_\Omega \B^{\Heis^2}f(q,p) = 2\B^{\Heis^2}f_\Omega(q,p)$.
\end{proof}

\section{A model for activated regions}\label{sec:model}

% The primary visual cortex V1 shows a remarkable organization of neurons devoted to specific sensory measurements and axonal connections that allow elaborated perceptive tasks. Both are involved in the two dimensional maps of orientations preference (OPM) shown in Fig.\ref{pinwheels} left, which have been measured by means of ad hoc in-vivo experiments with the method of gratings \cite{BG,Bl}. These experiments are designed to activate specific cortical regions corresponding to populations of cells with similar preferred orientation and visualize them with optical imaging techniques, and their results has been lately validated with single-cell precision \cite{Ohki}. We recall here the model presented in \cite{BCSS} that is aimed to describe and reproduce the images obtained by the gratings experiments in terms of the $SE(2)$-Bargmann transform (\ref{eq:SE2Barg}), hence providing a geometric interpretation for the OPM.

The primary visual cortex V1 shows a remarkable organization of neurons devoted to specific sensory measurements and axonal connections that allow elaborated perceptive tasks. Both are involved in the maps of orientations preference (OPM) shown in Fig.\ref{pinwheels}, which have been measured by means of in-vivo experiments designed to activate specific cortical regions corresponding to populations of cells with similar preferred orientation \cite{BG,Bl}. We recall here the model presented in \cite{BCSS} that is able to describe and reproduce the images obtained by these experiments in terms of the $SE(2)$-Bargmann transform (\ref{eq:SE2Barg}), hence providing a geometric interpretation for the OPM.

\subsection{The functional architecture of V1}

% The activation of retinal ganglion cells produce signals which are sent through brain circuitry to a sensory area located in the occipital lobe of the brain called the visual cortex, that is commonly partitioned into several distinct areas on the basis of anatomical and functional properties. V1 is the largest one, and probably the best studied cortical area of the brain \cite{HW,Braitenberg,Swindale1982,Daugman,BG,Bl,NW,WG,Ringach,CS,Ohki,HFC,KSLCWW}.

V1 is the first cortical area which elaborates the visual signal, sent by the retina. To each V1 cell there corresponds a receptive field, a region of the visual field that, when stimulated, will cause a neuron response. V1 is organized in a so-called retinotopic way, meaning that the centers of receptive fields relative to nearby neurons are correspondingly close. Moreover, nearby receptive fields are highly overlapping, providing a complete covering of the visual field. Since the works of Hubel and Wiesel in the 60's, V1 cells are known to be selective for local orientations of the visual stimuli, i.e. the main direction that characterizes retinal images restricted to the receptive field.
% Another crucial feature is that V1 is structured as a stack of two dimensional layers, and the displacement of cells with respect to their preferred orientation is the same on each layer \cite{HW}.
More recently, orientation selectivity has been described in a much deeper detail, the response of V1 simple cells being modeled as a Gabor filtering in \cite{Daugman}, and this assumption has been lately validated experimentally \cite{Ringach}. Given the input $f \in L^2(\R^2)$ as a function on the retina, cell response can then be assumed to be given by $F(q,p)$ as in (\ref{eq:csa}). The variable $q$ denotes the center of a cell receptive field and $p$ denotes a wave vector whose polar angle $\theta$ characterizes the cell orientation selectivity. This response is a complex function, whose real and imaginary parts describe respectively the output of two different population of cells, called even and odd cells \cite{Ringach}.
% The retinotopic organization can be intended as a diffeomorphism from the retina to V1, that maps a retinal position $q$ from to a position on one two dimensional layer of the cortex \cite{HW}. As a local approximation, we will consider this as a linear mapping, as well as we will consider the retina and V1 as flat surfaces, so up to a scale factor we can intend $q \in \R^2$ as the position on V1 of a given cell. We have to note that also the size of a receptive profile, corresponding to the parameter $\sigma$ in (\ref{eq:csa}), is characteristic of each neuron. It is then an additional approximation the assumption that we will adopt of the scale $\sigma$ as an intrinsic constant parameter \cite{BCSS}.
As a local approximation, we will consider the retina and V1 as flat surfaces, and the retinotopic organization as the identity map \cite{HW}, so that we can intend $q \in \R^2$ as the position on V1 of a given cell. We will also consider, in an additional approximation, that the size of a receptive field, corresponding to the parameter $\sigma$ in (\ref{eq:csa}), is an intrinsic constant parameter \cite{BCSS}.
From the point of view of orientation selectivity, V1 can be modeled \cite{PT, CS} as the 3-dimensional fiber bundle $\R^2_q \times \Pp_\theta$, where each point $(q, \theta)$ corresponds to a specific cell. The concrete two dimensional organization corresponds to a section that associates to a cell located at $q$ a preferred orientation $\theta$, and is commonly referred to as OPM (see Fig.\ref{pinwheels}, left).

Axonal connections in V1 can be broadly classified into two main classes: local connections, among nearby neurons, and long-range connections. In \cite{CS} it has been proposed a geometric model for long-range connections in the space $\R^2_q \times S^1_\theta$ as an implementation of the geometry of the $SE(2)$ group, that considers two cells $(q,\theta)$ and $(q',\theta')$ as connected by curves (\ref{expSE2}). This approach has shown effectiveness in reproducing different tasks, described by the Gestalt principle of good continuation. Moreover, the connectivity of V1 provides links between far away cells according to their preferred orientation \cite{BZSF}, showing then the relation with OPM.

% In many species, including primates, OPM assume a characteristic structure of an almost everywhere smooth displacement of a $\Pp$ quantity on a two dimensional layer. Due to the topological structure of $\Pp$, such a mapping will most probably possess singularities, which have been anticipated as a conjecture \cite{Braitenberg, Swindale1982} and lately discovered by measurements \cite{BG,Bl}, and several geometric analyses have been proposed concerning their formation in terms of self-organization and in the framework of the physics of defects \cite{WG}. We must note also that the presence of spontaneous neural activities in the early development of the visual cortex plays a determinant role in the formation of OPM. This is mainly due to the so-called retinal random waves \cite{HFC}, that are wavefronts randomly propagating in the retina. Each visual cortex is then subject to a randomness that is characteristic of the specific animal, which is then coded in its realization of OPM.

Several geometric analyses have been proposed concerning the formation of OPM in terms of self-organization and in the framework of the physics of defects \cite{WG}. We must note also that the presence of spontaneous neural activities in the early development of the visual cortex plays a determinant role in the formation of OPM. This is mainly due to the so-called retinal random waves \cite{HFC}, that are wavefronts randomly propagating in the retina. Each visual cortex is then subject to a randomness that is characteristic of the specific animal, which is then coded in its realization of OPM.

Some crucial principles of organizations of these structures were pointed out in \cite{NW}, noting in particular that the Fourier spectrum of the orientation maps is approximately concentrated on a circle (see Fig.\ref{pinwheels}, right), or equivalently that these aperiodic structures possess isotropic internal correlations at a fixed characteristic length. A very recent analysis \cite{KSLCWW} quantifies these observations with much sharper instruments, and also include them into a generative model. In \cite{NW} it was also introduced an empirical method able to reproduce orientation map-like structures as a superposition of complex plane waves with random phases: indeed, if $\phi_\varphi$ is a white noise with values in $[0,2\pi]$ and indexed by $\varphi \in \Pp$, then the structures in Fig.\ref{pinwheels} can be approximately reproduced as
\begin{equation}\label{eq:randomphases}
\theta(q) = \frac 12 \arg \int_0^\pi e^{i 2 \varphi} \cos\left(\Omega(q_1\cos\varphi + q_2\sin\varphi) + \phi_\varphi\right) d\varphi 
\end{equation}
where the integral should properly be intended in It\^o sense.

\begin{figure} 
\centering
\includegraphics[width=.6\textwidth]{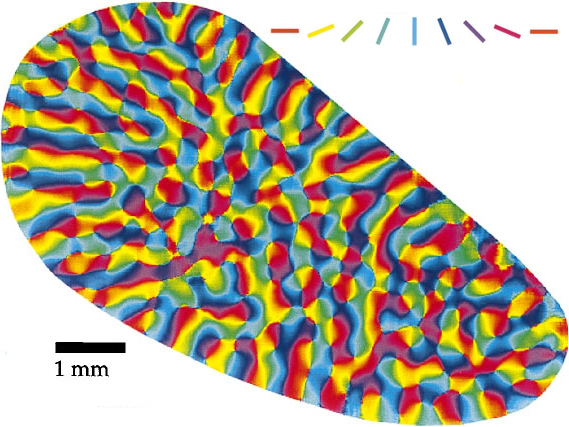} \quad 
\includegraphics[width=.35\textwidth,height=.35\textwidth]{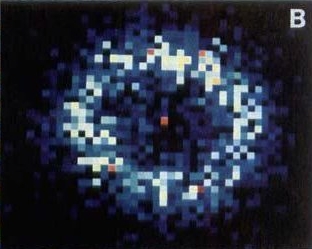}
\caption{Left: orientation preference maps in tree-shrew, coded with periodic colors (extracted from \cite{BZSF}). Right: power spectrum of orientation maps in macaque (extracted from \cite{NW})}\label{pinwheels}
\end{figure}

% The experiments that allow to obtain OPM as in Fig.\ref{pinwheels}, left \cite{BG,Bl,BZSF} rely on optical imaging techniques that quantify blood charge in the neural tissue using fMRI. This setting is then used to measure the activity in V1 caused by cells responses to so-called gratings. Gratings are images constituted by straight parallel black and white stripes shifting along the perpendicular direction, that can be easily provided by plane waves with phase shift. Several experiments with gratings have shown that the activated regions of V1 do not depend on the distances between stripes (the spatial frequency of plane waves), nor on the velocity of the shift, but only on the angle at which they are presented. Due to the phase shift, it is more correct to speak about orientation, since the stimulus, and consequently the resulting activity, can not be distinguished at angles of $\pi$. An example of the images obtained is given in Fig.\ref{blobcomparison}, left. The result of this experiment is then a family of real maps $\{A_\alpha(q)\}_{\alpha \in \Pp}$, that can be used to obtain the OPM coded with colors as in Fig.\ref{pinwheels} by performing a vector sum and then considering the resulting orientation \cite{BG}, as depicted in Fig.\ref{blobcomparison}, right
% \begin{equation}\label{eq:colorcoding}
% \theta(q) = \frac 12 \arg \int_0^\pi e^{i2\alpha} A_\alpha(q) d\alpha\ .
% \end{equation}

The experiments that allow to obtain OPM as in Fig.\ref{pinwheels}, left \cite{BG,Bl,BZSF} rely on optical imaging techniques that quantify blood charge in the neural tissue using fMRI. This setting is then used to measure the activity in V1 caused by cells responses to so-called gratings. Gratings are images constituted by straight parallel black and white stripes shifting along the perpendicular direction, that can be easily provided by plane waves with phase shift. Experiments with gratings have shown that activated regions depend only on the angle at which they are presented, or better the orientation, since the stimulus, and consequently the resulting activity, can not be distinguished at angles of $\pi$ due to the phase shift. An example of the images obtained is given in Fig.\ref{blobcomparison}, left. The result of this experiment is then a family of real maps $\{A_\alpha(q)\}_{\alpha \in \Pp}$, that can be used to obtain the OPM coded with colors as in Fig.\ref{pinwheels} by performing a vector sum and then considering the resulting orientation \cite{BG} (see Fig.\ref{blobcomparison}, right):
\begin{equation}\label{eq:colorcoding}
\theta(q) = \frac 12 \arg \int_0^\pi e^{i2\alpha} A_\alpha(q) d\alpha\ .
\end{equation}

We note that, starting from a given an orientation map $\theta(q)$, a way to obtain activated regions that are compatible with (\ref{eq:colorcoding}) is by scalar product, i.e.
\begin{equation}\label{eq:scalarprcol}
A_\alpha(q) = \Re\left(e^{-i2\alpha}e^{i2\theta(q)}\right)
\end{equation}
so in particular if we consider the construction (\ref{eq:randomphases}), then (\ref{eq:scalarprcol}) reduces to
\begin{equation}\label{eq:blobsempirical}
A_\alpha(q) = \int_0^\pi \cos\left(\Omega(x\cos\varphi + y\sin\varphi) + \phi(\varphi)\right) \cos(2(\varphi-\alpha)) d\varphi\ .
\end{equation}
We are going to see that model introduced in \cite{BCSS} does indeed reproduce activities in the form (\ref{eq:blobsempirical}), and hence OPM in the form (\ref{eq:randomphases}), but starting from a geometric model of the activities, so that OPM arise as a consequence of the color coding (\ref{eq:colorcoding}).

\subsection{Reproducing activated regions}

The model we have proposed in \cite{BCSS} in order to reproduce the activity patterns resulting from the gratings experiment can be stated in terms of an $SE(2)$-Bargmann transform of a specific white noise process. This will be properly symmetrized due to the intrinsic characteristics of the patterns, and the presence of randomness will be motivated in terms of the retinal random waves previously described. The geometry of the different activities resulting from the exposure to gratings at various orientations is then motivated in terms of the uncertainty principle, providing a description of a family of functions indexed by orientations in terms of a single function on the group.

\paragraph{Statement of the model}

Given a white noise $\phi_\varphi$ with values in $[0,2\pi]$ and indexed by $\varphi \in [0,\pi]$, and considering a prolongation to $\varphi \in [0,\pi]$ such that $\phi_{\varphi + \pi} = -\phi_{\varphi}$, we define the functions
\begin{displaymath}
F_\Omega^\lambda(q,\theta) \doteq \B_{\Omega}^{\lambda}e^{i\phi}(q,\theta) = \int_0^{2\pi} e^{i\Omega(q_1\cos\varphi + q_2\sin\varphi)} e^{\lambda\Omega\cos(\varphi - \theta)}e^{i\phi_{\varphi}} d\varphi
\end{displaymath}
recalling that they are minimal uncertainty states in the sense of Proposition \ref{prop:minuncmix}.

Starting from them we define the activity functions $\big\{A_\theta : \R^2 \rightarrow \R \big\}_{\theta \in \Pp}$ as
\begin{eqnarray}\label{modelconstr}
A_\theta(q) & \doteq & F_\Omega^\lambda(q,\theta) + F_\Omega^\lambda(q,\theta+\pi) - F_\Omega^\lambda(q,\theta+\frac{\pi}{2}) + F_\Omega^\lambda(q,\theta+\frac{3}{2}\pi)\nonumber\\
& = & \Re \left( F_\Omega^\lambda(q,\theta) - F_\Omega^\lambda(q,\theta+\frac{\pi}{2})\right)\ .
\end{eqnarray}
These functions are indeed $\pi$ periodic in $\theta$, to represent orientations, and provide opposite response at orthogonal angles: $A_{\theta+\frac{\pi}{2}} = -A_{\theta}$, as it is the case for V1 cells.

By direct computation we can explicitly write (\ref{modelconstr}) in the following form
\begin{lem}
Calling
\begin{equation}\label{effpot}
V_\Omega^\lambda(\varphi) = \cosh\left(\lambda \Omega \cos\varphi\right) - \cosh\left(\lambda \Omega\sin\varphi\right)
\end{equation}
then (\ref{modelconstr}) reads
\begin{equation}\label{model}
A_\theta(q) = \int_0^\pi \cos\left(\Omega(q_1\cos\varphi + q_2\sin\varphi)+\phi(\varphi)\right) V_\Omega^\lambda(\varphi - \theta) d\varphi
\end{equation}
\end{lem}

We note that, if we consider a real retinal image obtained as a superposition of plane random waves \cite{HFC} at a fixed wavelength as given by
\begin{displaymath}%\label{eq:input}
f(x) = \int_0^\pi \cos(\Omega(x_1\cos\varphi + x_2\sin\varphi) + \phi(\varphi)) d\varphi
\end{displaymath}
then by Theorem \ref{Teo} the function $F_\Omega^\lambda$ can be obtained as the the resulting cell response in the form (\ref{eq:csa}). This in particular motivates the choice of random phases $\phi_\varphi$.

\paragraph{Comparison with the experiments}

In Fig.\ref{blobcomparison} we show a comparison with the experiments. The parameters were chosen as $\lambda \Omega \approx 1$, providing an approximate equipartition of uncertainty \cite{CN}. However the results are stable under small variations of $\lambda \Omega$. This is reasonable, since the functions (\ref{effpot}) are such that
\begin{displaymath}
V_\Omega^\lambda(\varphi) \approx \cos(2\varphi) \quad \textrm{for} \ \lambda\Omega \lesssim 1
\end{displaymath}
up to a multiplicative constant. This ensures in particular that the expression (\ref{eq:blobsempirical}) can be considered as an approximation to (\ref{model}).

\begin{figure}[h!]
\includegraphics[width=.31\textwidth]{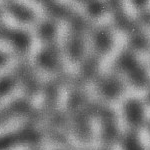}
\includegraphics[width=.31\textwidth]{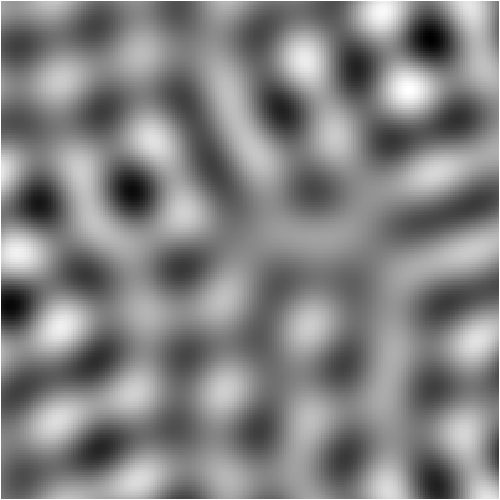}
\includegraphics[width=.35\textwidth,height=.31\textwidth]{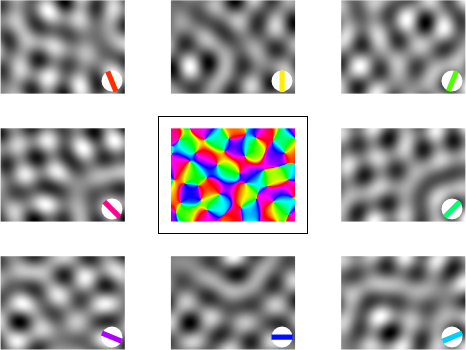}
\caption{Left, the configurations of orientation activity maps as experimentally observed (extracted from \cite{BZSF}). Center, a sample of the activity maps $A_\theta$ produced by (\ref{model}). Right, gray valued images visualize the maps $A_\theta$ by varying the angle $\theta$. The color image has been constructed associating a color coding representation to preferred orientations, as in (\ref{eq:colorcoding}). (Center and right figures extracted from \cite{BCSS})}\label{blobcomparison}
\end{figure}

% This model allows to obtain the construction (\ref{eq:randomphases}), hence providing an interpretation of it in terms of the group $SE(2)$.
% $$
% \begin{array}{rcl}
% z(q) & = & \displaystyle{\int_0^\pi e^{i2\theta} \left(\int_0^\pi \cos\left(\Omega(q_1\cos\varphi + q_2\sin\varphi) + \phi(\varphi)\right) \psi(\varphi - \theta) d\varphi \right) d\theta}\vspace{4pt}\\
% % & = & \int_0^\pi \cos\left(\Omega(q_1\cos\varphi + q_2\sin\varphi) + \phi(\varphi)\right) \left(\int_0^{\pi} e^{i2\alpha}\psi(\varphi - \theta) d\alpha\right) d\varphi\\
% & = & C[\psi] \displaystyle{\int_0^\pi  e^{i 2 \varphi} \cos\left(\Omega(x\cos\varphi + y\sin\varphi) + \phi(\varphi)\right) d\varphi}\ .
% \end{array}
% $$

% indeed
% \begin{eqnarray*}
% f(\theta) & = & \left(\frac{\mu^2}{2} + \frac{\mu^4}{4!}\right)\cos(2\theta) + \sum_{n=3}^\infty \frac{\mu^{2n}}{(2n)!}(\cos^{2n}\theta - \sin^{2n}\theta)\\
% & = & \left(\frac{\mu^2}{2} + \frac{\mu^4}{4!} + \sum_{n=3}^\infty \frac{\mu^{2n}}{(2n)!} \sum_{k=0}^{n-1} (\cos^2\theta)^{n-k}(\sin^2\theta)^k\right)\cos(2\theta)\\
% & = & \left(\frac{\mu^2}{2} + \frac{\mu^4}{4!} + \sum_{n=3}^\infty \frac{\mu^{2n}}{(2n)!} \sum_{k=0}^{n-1} (\cos^2\theta)^{n-k}(1-\cos^2\theta)^k\right)\cos(2\theta)\\
% & = & \left(\frac{\mu^2}{2} + \frac{\mu^4}{4!} + \sum_{n=3}^\infty \frac{\mu^{2n}}{(2n)!} \sum_{k=0}^{n-1} \sum_{l = 0}^k \binom{k}{l}(\cos^2\theta)^{n+l-k} \right)\cos(2\theta)
% \end{eqnarray*}

\section{Conclusions}

The introduced notion of $SE(2)$-Bargmann space allows to pass the problem of nonintegrability of the representation using a measure that is singular in the Fourier domain, hence reflecting the behavior of the irreducible representation. This method is constructive, and permits to perform integrations on the group in terms of equivalence classes. Moreover, the choice of a fiducial vector as a minimum of the uncertainty principle provides a relation between two symmetries that appeared unrelated as $SE(2)$ and $\Heis^2$. This relation is given at the level of coherent states transforms, and relies on a compatibility of the complex structures associated to the coherent states of the two Lie groups, that can then be considered nested one into the other from this perspective. Moreover, it allows to complete the concrete construction of the space of surjectivity of the $SE(2)$-Bargmann transform.

This approach unifies two main symmetries present in the maps of orientation preference of the primary visual cortex, and allows to produce a model that is able to reproduce neural maps of activity measured experimentally. The model is based on the uncertainty principle of $SE(2)$, describing activated regions in terms of minimal uncertainty states. This uncertainty principle acts at the macroscopic level, hence does not rely on any microscopic physics assumption on the brain, but rather refers to functional features of the cortex. The modularity of the Lie group approach allows to extend the models to other higher symmetries characterizing the functional architecture, as in \cite{SCP}, and in perspective to model high level functionality of vision.

\end{document}